\documentclass[11pt,reqno]{amsart}
\usepackage{amsmath}
\usepackage{amstext}
\usepackage{amsbsy}
\usepackage{amsopn}
\usepackage{upref}
\usepackage{amsthm}
\usepackage{amsfonts}
\usepackage{amssymb}
\usepackage{mathrsfs}
\usepackage{times}
\allowdisplaybreaks
 \newtheorem{theorem}{Theorem}
 
 \newtheorem{proposition}[theorem]{Proposition}
 \newtheorem{lemma}[theorem]{Lemma}
\theoremstyle{definition}
 
\theoremstyle{remark}
 
\newcommand{\ep}{\varepsilon}
\newcommand{\p}{\partial}
\begin{document}
\title[Dispersive Flow]{A third order dispersive flow for closed curves
\\ into almost Hermitian manifolds}
\author[H.~Chihara and E.~Onodera]{Hiroyuki Chihara and Eiji Onodera}
\address[Hiroyuki Chihara]{Mathematical Institute, Tohoku University, Sendai 980-8578, Japan}
\email{chihara@math.tohoku.ac.jp}
\address[Eiji Onodera]{Mathematical Institute, Tohoku University, Sendai 980-8578, Japan}
\email{sa3m09@math.tohoku.ac.jp}
\thanks{HC is supported by JSPS Grant-in-Aid for Scientific Research \#20540151.}
\thanks{EO is supported by JSPS Fellowships for Young Scientists and JSPS Grant-in-Aid for Scientific Research \#19$\cdot$3304.}
\begin{abstract}
We discuss a short-time existence theorem of solutions to 
the initial value problem for a third order dispersive flow 
for closed curves into a compact almost Hermitian manifold. 
Our equations geometrically generalize a physical model 
describing the motion of vortex filament. 
The classical energy method cannot work for this problem 
since the almost complex structure of the target manifold is not supposed 
to be parallel with respect to the Levi-Civita connection. 
In other words, a loss of one derivative arises 
from the covariant derivative of the almost complex structure. 
To overcome this difficulty, 
we introduce a bounded pseudodifferential operator 
acting on sections of the pullback bundle, 
and eliminate the loss of one derivative from 
the partial differential equation of the dispersive flow. 
\end{abstract}
\maketitle
\section{Introduction}
\label{section:introduction}
Let $(N,J,h)$ be a $2n$-dimensional compact almost Hermitian manifold 
with an almost complex structure $J$ and a Hermitian metric $h$. 
Consider the initial value problem for a third order dispersive flow of the form 
\begin{align}
  u_t
  =
  a\nabla_x^2u_x
  +
  J_u\nabla_xu_x
  +
  bh(u_x,u_x)u_x
& \quad\text{in}\quad
  \mathbb{R}{\times}\mathbb{T},
\label{equation:pde}
\\
  u(0,x)
  =
  u_0(x)
& \quad\text{in}\quad
  \mathbb{T},
\label{equation:data}
\end{align}
where 
$u$ is an unknown mapping of $\mathbb{R}\times\mathbb{T}$ to $N$, 
$(t,x)\in\mathbb{R}\times\mathbb{T}$, 
$\mathbb{T}=\mathbb{R}/\mathbb{Z}$, 
$u_t=du(\p/\p{t})$, 
$u_x=du(\p/\p{x})$, 
$du$ is the differential of the mapping $u$, 
$u_0$ is a given closed curve on $N$, 
$\nabla$ is the induced connection, 
$a,b\in\mathbb{R}$ are constant. 
$u(t)$ is a closed curve on $N$ for fixed $t\in\mathbb{R}$, 
and $u$ describes the motion of a closed curve subject to 
\eqref{equation:pde}. 
We present local expression of the covariant derivative $\nabla_x$. 
Let $y^1,\dotsc,y^{2n}$ be local coordinates of $N$, 
and let $h=\sum_{a,b=1}^{2n}h_{ab}dy^a{\otimes}dy^b$. 
We denote by $\Gamma^a_{bc}$, $a,b,c=1,\dotsc,2n$,  
the Christoffel symbol of $(N,J,h)$. 
For a smooth closed curve $u:\mathbb{T}\rightarrow{N}$, 
$\Gamma(u^{-1}TN)$ is the set of all smooth sections 
of the pullback bundle $u^{-1}TN$. 
If we express $V{\in}\Gamma(u^{-1}TN)$ as
$$
V(x)=\sum_{a=1}^{2n}V^a(x)\left(\frac{\p}{\p{y^a}}\right)_u, 
$$
then, $\nabla_xV$ is given by  
$$
\nabla_xV(x)
=
\sum_{a=1}^{2n}
\left\{
\frac{{\p}V^a}{\p{x}}(x)
+
\sum_{b,c=1}^{2n}
\Gamma^a_{bc}\bigl(u(x)\bigr)V^b(x)\frac{\p{u^c}}{\p{x}}(x)
\right\}
\left(\frac{\p}{\p{y^a}}\right)_u.
$$
\par
The equation \eqref{equation:pde} geometrically generalizes
two-sphere valued partial differential equations 
modeling the motion of vortex filament. 
In his celebrated paper \cite{darios}, 
Da Rios first formulated the motion of vortex filament as 
\begin{equation}
\vec{u}_t=\vec{u}\times\vec{u}_{xx},
\label{equation:darios}
\end{equation}
where 
$\vec{u}=(u^1,u^2,u^3)$ is an $\mathbb{S}^2$-valued function of $(t,x)$, 
$\mathbb{S}^2$ is a unit sphere in $\mathbb{R}^3$ 
with a center at the origin, 
and $\times$ is the exterior product in $\mathbb{R}^3$. 
The physical meanings of $\vec{u}$ and $x$ are 
the tangent vector and the signed arc length of vortex filament respectively. 
After eighty five years, a modified model equation of
vortex filament 
\begin{equation}
\vec{u}_t=\vec{u}\times\vec{u}_{xx}
+
a
\left[
\vec{u}_{xxx}
+
\frac{3}{2}
\left\{\vec{u}_x\times(\vec{u}\times\vec{u}_x)\right\}_x
\right]
\label{equation:FM}
\end{equation}
was proposed by Fukumoto and Miyazaki in \cite{FM}. 
When $a,b=0$, \eqref{equation:pde} generalizes \eqref{equation:darios} 
and solutions to \eqref{equation:pde} are called one-dimensional Schr\"odinger maps. 
When $b=a/2$, \eqref{equation:pde} generalizes \eqref{equation:FM}. 
\par
In recent ten years, physical models such as 
\eqref{equation:darios} and \eqref{equation:FM} 
have been generalized and studied 
from a point of view of geometric analysis in mathematics. 
The relationship between 
the geometric settings and 
the structure of such partial differential equations and their solutions 
has been recently investigated in mathematics. 
\par
The reduction of equations to simpler ones leads us to rough
understandings of their structure. 
This idea originated from Hasimoto's transform discovered in \cite{hasimoto}. 
In their pioneering work \cite{CSU}, 
Chang, Shatah and Uhlenbeck first rigorously studied the PDE structure of 
\eqref{equation:pde} when $a=b=0$, $x\in\mathbb{R}$, 
and $(N,J,h)$ is a compact Riemann surface. 
They constructed a good moving frame along the map and 
reduced \eqref{equation:pde} to a simple complex-valued 
semilinear Schr\"odinger equation 
under the assumption that 
$u(t,x)$ has a fixed base point as  $x\rightarrow+\infty$. 
Similarly, Onodera reduced 
\eqref{equation:pde} with $a\ne0$ 
and a one-dimensional fourth order dispersive flow 
to complex-valued equations in \cite{onodera2}. 
Generally speaking, these reductions require some restrictions on 
the range of the mappings, and one cannot make use of them 
to solve the initial value problem for the original equations 
without restrictions on the range of the initial data. 
\par
How to solve the initial value problem for 
such geometric dispersive equations 
is a fundamental question. 
In his pioneering work \cite{koiso}, 
Koiso first reformulated \eqref{equation:darios} geometrically, 
and proposed the equation \eqref{equation:pde} 
with $a,b=0$ and the K\"ahler condition $\nabla^NJ=0$, 
where $\nabla^N$ is the Levi-Civita connection of $(N,J,h)$. 
Moreover, Koiso established the standard short-time existence theorem, 
and proved that if $(N,J,h)$ is locally symmetric, 
that is, $\nabla^NR=0$, then the solution exists globally in time, 
where $R$ is the Riemannian curvature tensor of $N$. 
See \cite{PWW} also for one-dimensional Sch\"odinger maps. 
Recently, Onodera studied local and global existence theorems 
of \eqref{equation:pde}-\eqref{equation:data} 
in case $a\ne0$ in \cite{onodera1} and \cite{onodera3}. 
To be more precise, 
\cite{onodera1} studied the case $\nabla^NJ=0$, 
and proved a short-time existence theorem. 
Moreover, he proved that 
if $(N,J,h)$ is a compact Riemann surface 
with a constant sectional curvature $K$ 
and a condition $b=Ka/2$ is satisfied, 
then the time-local solution can be extended globally in time. 
Nishiyama and Tani proved the global existence of solutions to 
the initial value problem for \eqref{equation:FM} 
in \cite{NT} and \cite{TN}. 
Since $K=1$ for $N=\mathbb{S}^2$, 
the global existence theorem in \cite{onodera1} 
is the generalization of the results \cite{NT} and \cite{TN}. 
\cite{onodera3} studied a short-time existence theorem 
for \eqref{equation:pde}-\eqref{equation:data} 
in case that $(N,J,h)$ is a compact almost Hermitian manifold 
and $x\in\mathbb{R}$. 
Being inspired by Tarama's beautiful results on the characterization 
of $L^2$-well-posedness of the initial value problem 
for a one-dimensional linear third order dispersive equations in \cite{tarama2} 
(See also \cite{mizuhara}), 
Onodera introduced a gauge transform on the pullback bundle 
to make full use of so-called local smoothing effect of 
$e^{t\p^3/\p{x^3}}$, 
and proved a short-time existence theorem. 
\par
Both of the reduction of equations and the study of existence theorem 
are deeply connected with the relationship between 
the geometric settings of equations and 
the theory of linear dispersive partial differential equations. 
For the latter subject, see, e.g., 
\cite{chihara2}, 
\cite{doi}, 
\cite[Lecture VII]{mizohata}, 
\cite{mizuhara}, 
\cite{tarama1}, 
\cite{tarama2} 
and references therein. 
Being concerned with the compactness of the source space, 
we need to mention local smoothing effect of 
dispersive partial differential equations. 
It is well-known that solutions to the initial value problem for some 
kinds of dispersive equations gain extra smoothness in comparison with 
the initial data. 
In his celebrated work \cite{doi}, 
Doi characterized the existence of microlocal smoothing effect of
Schr\"odinger evolution equations on complete Riemannian manifolds 
according to the global behavior of the geodesic flow on the unit
cotangent sphere bundle over the source manifolds. 
Roughly speaking, 
the local smoothing effect occurs if and only if 
all the geodesics go to ``infinity''. 
For more general dispersive equations, 
the existence or nonexistence of local smoothing effect 
is determined by the global behavior of the Hamilton flow 
generated by the principal symbol of the equations. 
In particular, if the source space is compact, 
then no smoothing effect occurs 
since all the integral curves of the Hamilton vector field are trapped. 
For this reason, it is essential to study the initial value problem 
\eqref{equation:pde}-\eqref{equation:data} 
when the source space is $\mathbb{T}$ and not $\mathbb{R}$. 
\par
Here we mention the relationship between 
the K\"ahler condition $\nabla^NJ=0$ 
and the structure of the equation \eqref{equation:pde}. 
All the preceding works on \eqref{equation:pde} 
except for \cite{onodera3} 
assume that $(N,J,h)$ is a K\"ahler manifold. 
If $\nabla^NJ=0$, then \eqref{equation:pde} behaves like 
symmetric hyperbolic systems, 
and the short-time existence theorem can be proved by the classical energy method. 
See \cite{onodera1} for the detail. 
If $\nabla^NJ\ne0$, then \eqref{equation:pde} has a first 
order terms in some sense, and the classical energy method breaks down. 
\par
The purpose of the present paper is to show 
a short-time existence theorem for
\eqref{equation:pde}-\eqref{equation:data} 
without using the K\"ahler condition and the local smoothing effect. 
To state our results, we here introduce function spaces of mappings.  
For a nonnegative integer $k$, 
$H^{k+1}(\mathbb{T};TN)$ is the set of all continuous mappings 
$u:\mathbb{T}{\rightarrow}N$ satisfying 
$$
\lVert{u}\rVert_{H^{k+1}}^2
=
\sum_{l=0}^k
\int_\mathbb{T}
h\left(\nabla_x^ku_x,\nabla_x^ku_x\right)
dx
<\infty,
$$
See e.g., \cite{hebey} for the Sobolev space of mappings.  
The Nash embedding theorem shows that 
there exists an isometric embedding 
$w{\in}C^\infty(N;\mathbb{R}^d)$ with some integer $d>2n$. 
See \cite{GR}, \cite{gunther} and \cite{nash} 
for the Nash embedding theorem. 
Let $I$ be an interval in $\mathbb{R}$. 
We denote by $C(I;H^{k+1}(\mathbb{T};TN))$ the set of all 
$H^{k+1}(\mathbb{T};TN)$-valued continuous functions on $I$, 
In other words, we define it 
by the pullback of the function space as 
$C(I;H^{k+1}(\mathbb{T};TN))=C(I;{w^\ast}H^{k+1}(\mathbb{T};\mathbb{R}^d))$, 
where $H^{k+1}(\mathbb{T};\mathbb{R}^d)$ is the usual Sobolev space 
of $\mathbb{R}^d$-valued functions on $\mathbb{T}$. 
\par
Here we state our main results. 
\begin{theorem}
\label{theorem:main}
Let $k$ be a positive integer satisfying $k\geqslant4$. 
Then, for any $u_0{\in}H^{k+1}(\mathbb{T};TN)$, there exists 
$T=T(\lVert{u}\rVert_{H^{5}})>0$ such that 
{\rm \eqref{equation:pde}-\eqref{equation:data}} 
possesses a unique solution 
$u{\in}C([-T,T];H^{k+1}(\mathbb{T};TN))$.  
\end{theorem}
We will prove Theorem~\ref{theorem:main} by 
the uniform energy estimates of solutions to 
a fourth order parabolic regularized equation.  
To avoid the difficulty arising from $\nabla_xJ_u$,  
we modify the method introduced for the initial value problem for 
Schr\"odinger maps of a closed Riemannian manifold to 
a compact almost Hermitian manifold in \cite{chihara3}. 
Being inspired by his own previous paper \cite{chihara1}, 
Chihara introduced a transformation of unknown mappings defined by 
a bounded pseudodifferential operator 
acting on sections of $\Gamma(u^{-1}TN)$, 
and eliminated first order terms coming from $\nabla_x{J_u}$ 
in \cite{chihara3}. 
\par
The plan of the present paper is as follows. 
Section~\ref{section:tarama} studies the well-posedness of 
an auxiliary initial value problem for 
some one-dimensional linear dispersive partial differential equations
related with \eqref{equation:pde}-\eqref{equation:data}. 
We believe that Section~\ref{section:tarama} will be very helpful 
to understand our idea of the proof of Theorem~\ref{theorem:main}, 
though the arguments and results there are nonsense from a point of view of 
the theory of linear partial differential equations. 
Section~\ref{section:proof} proves Theorem~\ref{theorem:main}. 
\section{An Auxiliary Linear Problem}
\label{section:tarama}
In this section we study the initial value problem for a one-dimensional
third order linear dispersive partial differential equation 
related with \eqref{equation:pde} of the form
\begin{alignat}{2}
  LU
  \equiv
  U_t+U_{xxx}+\sqrt{-1}\{a(x)U_x\}_x+b_x(x)U_x+c(x)U
& =
  F(t,x)
& 
  \quad\text{in}\quad
& \mathbb{R}{\times}\mathbb{T},
\label{equation:pde1}
\\
  U(0,x)
& =
  U_0(x)
& 
  \quad\text{in}\quad
& \mathbb{T},
\label{equation:data1}
\end{alignat}
where 
$U$ is a complex-valued unknown function of 
$(t,x)\in\mathbb{R}\times\mathbb{T}$, 
$a,b,c{\in}C^\infty(\mathbb{T})$, 
$\operatorname{Im} a=0$, 
$U_0(x)$ and $F(t,x)$ are given functions. 
The operator $L$ is very special in the sense that 
the coefficient of the first order term is 
a derivative of a smooth function. 
The well-posedness of the initial value problem for 
third and fourth order dispersive equations on $\mathbb{R}$ or $\mathbb{T}$ 
was studied in \cite{mizuhara}, \cite{tarama1} and \cite{tarama2}. 
In most of cases the well-posedness was characterized by the
conditions on the coefficients of differential operators. 
Let $L^2(\mathbb{T})$ be the standard Lebesgue space of 
square-integrable functions on $\mathbb{T}$, 
and let $L^1_{\text{loc}}(\mathbb{R};L^2(\mathbb{T}))$ 
be the set of all 
$L^2(\mathbb{T})$-valued locally integrable functions 
on $\mathbb{R}$. 
Mizuhara characterized the well-posedness 
of the initial value problem for a general third order dispersive
equations on $\mathbb{R}\times\mathbb{T}$. 
In view of his results in \cite[Theorem~6.1]{mizuhara}, 
one can immediately check that the special initial value problem 
\eqref{equation:pde1}-\eqref{equation:data1} is well-posed. 
\begin{proposition}
\label{theorem:mizuhara}
{\rm \eqref{equation:pde1}-\eqref{equation:data1}} is $L^2$-well-posed, 
that is, for any $U_0{\in}L^2(\mathbb{T})$ and 
for any $F{\in}L^1_{\text{loc}}(\mathbb{R};L^2(\mathbb{T}))$, 
{\rm \eqref{equation:pde1}-\eqref{equation:data1}} 
possesses a unique solution 
$U{\in}C(\mathbb{R};L^2(\mathbb{T}))$.  
\end{proposition}
All the descriptions in the present section are meaningless 
from a viewpoint of the general theory of 
linear partial differential equations. 
However, the purpose of this section is to illustrate 
our idea of the proof of Theorem~\ref{theorem:main} 
by showing the special proof of Proposition~\ref{theorem:mizuhara}. 
In what follows we make use of 
an elementary theory of pseudodifferential operators on $\mathbb{R}$. 
See \cite{kumano-go} for instance. 
In view of the idea in \cite[Section~2]{kpv}, 
one can deal with pseudodifferential operators on $\mathbb{T}$ 
in the same way as those on $\mathbb{R}$ 
without using the general theory of 
pseudodifferential operators on manifolds. 
$C^\infty(\mathbb{T})$ is regarded 
as the set of all $1$-periodic smooth functions on $\mathbb{R}$.  
Its topological dual is 
the set of all $1$-periodic tempered distributions on $\mathbb{R}$. 
\par
Let $p(\xi)$ be a real-valued smooth odd function on $\mathbb{R}$ 
satisfying $p(\xi)=1/\xi$ for $\xi\in\mathbb{R}\setminus(-2,2)$ 
and $p(\xi)=0$ for $\xi\in[-1,1]$. 
A pseudodifferential operator $p(D_x)$ is defined by 
an oscillatory integral of the form 
$$
p(D_x)u(x)
=
\frac{1}{2\pi}
\iint_{\mathbb{R}\times\mathbb{R}}
e^{\sqrt{-1}(x-y)\xi}
p(\xi)
U(y)
dyd\xi
\quad\text{for}\quad
U\in\mathscr{B}^\infty(\mathbb{R}),
$$
where 
$D_x=-\sqrt{-1}\p/\p{x}$, 
$\mathscr{B}^\infty(\mathbb{R})$ is the set of all
bounded $C^\infty$-functions on $\mathbb{R}$ 
whose derivative of any order is also bounded in $\mathbb{R}$. 
It is well-known that 
$p(D_x)$ is well-defined on $\mathscr{B}^\infty(\mathbb{R})$ 
and extended on the set of all tempered distributions on $\mathbb{R}$. 
$-\sqrt{-1}p(D_x)$ is an essential realization of 
the integral over $(-\infty,x]$ by pseudodifferential operators. 
The important properties of $p(D_x)$ are the following. 
\begin{lemma}
\label{theorem:psdo}
If $U(x)$ is real-valued and $1$-periodic, then 
so is $\sqrt{-1}p(D_x)U(x)$.  
\end{lemma}
\begin{proof}
Let $U{\in}C^\infty(\mathbb{T})$. 
We can easily check that $p(D_x)u(x)$ is $1$-periodic 
by using a translation $x{\mapsto}x+1$. 
Suppose that $U(x)$ is real-valued in addition. 
Then, 
\begin{align*}
  \operatorname{Im} 
  \left\{\sqrt{-1}p(D_x)U(x)\right\}
& =
  \operatorname{Re} 
  \left\{p(D_x)U(x)\right\}
\\
& =
  \frac{1}{2\pi}
  \iint_{\mathbb{R}\times\mathbb{R}}
  \operatorname{Re}
  \left\{
  e^{\sqrt{-1}(x-y)\xi}
  p(\xi)
  U(y)
  \right\}
  dyd\xi
\\
& =
  \frac{1}{2\pi}
  \iint_{\mathbb{R}\times\mathbb{R}}
  \cos\{(x-y)\xi\}
  p(\xi)
  U(y)
  dyd\xi=0
\end{align*}
since the integrand in the last integral above is an odd function in $\xi$. 
\end{proof}
Our special proof of Proposition~\ref{theorem:mizuhara} uses 
a bounded pseudodifferential operator defined by 
$$
\lambda(x,D_x)=1-\tilde{\lambda}(x,D_x), 
\quad
\tilde{\lambda}(x,\xi)
=
\frac{\sqrt{-1}}{3}b(x)p(\xi). 
$$
Roughly speaking, $\lambda(x,D_x)$ is 
a linear automorphism on $L^2(\mathbb{T})$. 
Indeed, it is easy to see that there exists a constant $M>1$ 
depending on $b(x)$ and $p(\xi)$ such that 
\begin{equation}
M^{-1}\lVert{U}\rVert
\leqslant
\mathcal{N}(U)
\leqslant
M\lVert{U}\rVert
\quad\text{for any}\quad 
U{\in}L^2(\mathbb{T}),
\label{equation:auto}
\end{equation}
where 
$\mathcal{N}(U)^2=\lVert\lambda(x,D_x)U\rVert^2+\lVert\langle{D_x}\rangle^{-1}U\rVert^2$, 
$\langle{D_x}\rangle=(1-\p^2/\p{x}^2)^{1/2}$, 
and 
$\lVert\cdot\rVert$ is the norm of $L^2(\mathbb{T})$. 
We prove Proposition~\ref{theorem:mizuhara} by using a transform 
$U\mapsto\lambda(x,D_x)U$ as follows. 
\begin{proof}[Sketch of proof of Proposition~\ref{theorem:mizuhara}] 
It suffices to show forward and backward energy inequalities. 
See \cite[Section~23.1]{hoermander} for instance. 
We obtain only an energy inequality in the positive direction in $t$. 
The backward one can be obtained similarly. 
A direct computation shows that 
\begin{align}
  \lambda(x,D_x)L
& =
  (\p_t+\p_x^3+\sqrt{-1}\p_xa(x)\p_x)\lambda(x,D_x)
\nonumber
\\
& \quad
  -
  [\tilde{\lambda}(x,D_x),\p_x^3]
  +
  b_x(x)\p_x+r_1(x,D_x),
\label{equation:saotome}
\\
  r_1(x,D_x)
& =
  -
  [\tilde{\lambda}(x,D_x),\sqrt{-1}\p_xa(x)p_x]
  -
  \tilde{\lambda}(x,D_x)b_x(x)\p_x
  +
  \lambda(x,D_x)c(x),
\nonumber
\\
  \langle{D_x}\rangle^{-1}L
& =
  (\p_t+\p_x^3+\sqrt{-1}\p_xa(x)\p_x)\langle{D_x}\rangle^{-1}
  +r_2(x,D_x)
\label{equation:hikaru}
\\
  r_2(x,D_x)
& =
  [\langle{D_x}\rangle^{-1},\sqrt{-1}\p_xa(x)\p_x]
  +
  \langle{D_x}\rangle^{-1}\bigl(b_x(x)\p_x+c(x)\bigr),
\nonumber
\end{align}
where $\p_t=\p/\p{t}$ and $\p_x=\p/\p{x}$. 
$r_1(x,D_x)$ and $r_2(x,D_x)$ are $L^2$-bounded pseudodifferential operators. 
We remark that 
$$
-[\tilde{\lambda}(x,D_x),\p_x^3]=-b_x(x)\p_x+r_3(x,D_x),
$$
$$
r_3(x,D_x)
=
b_x(x)\p_x(1-p(D_x)D_x)
-
b_{xx}(x)p(D_x)D_x
+
\frac{\sqrt{-1}}{3}b_{xxx}(x)p(D_x), 
$$
and $r_3(x,D_x)$ is also an $L^2$-bounded pseudodifferential operator. 
Set $r_4=r_1+r_3$ for short. 
Then, \eqref{equation:saotome} becomes 
\begin{equation}
\lambda(x,D_x)L
=
(\p_t+\p_x^3+\sqrt{-1}\p_xa(x)\p_x)\lambda(x,D_x)
+
r_4(x,D_x). 
\label{equation:otome} 
\end{equation}
Fix arbitrary $T>0$. 
Suppose that 
$U{\in}C([0,T];H^3(\mathbb{T})){\cap}C^1([0,T];L^2(\mathbb{T}))$. 
By using \eqref{equation:hikaru} and \eqref{equation:otome}, 
one can easily show that there exists a positive constant $C_0$ 
depending on $a$, $b$, $c$ and $p$ such that 
$$
\frac{d\mathcal{N}(U(t))^2}{dt}
\leqslant
C_0
\left(
\mathcal{N}(U(t))
+
\mathcal{N}(LU(t))
\right)
\mathcal{N}(U(t)),
$$
which implies a desired energy inequality
$$
\lVert{U(t)}\rVert
\leqslant
C_1
\left\{
\lVert{U(0)}\rVert
+
\int_0^t
\lVert{LU(s)}\rVert
ds
\right\}
\quad\text{for}\quad
t\in[0,T], 
$$
where $C_1$ is a positive constant 
depending only on $a$, $b$, $c$ and $p$. 
\end{proof}
%
%
\section{Proof of Theorem~\ref{theorem:main}}
\label{section:proof}
We shall prove Theorem~\ref{theorem:main} by the uniform energy
estimates of solutions to the initial value problem for 
semilinear parabolic equations of the form
\begin{align}
  u^\ep_t
  =
  -
  \ep\nabla_x^3u^\ep_x
  +
  a\nabla_x^2u^\ep_x
  +
  J_{u^\ep}\nabla_xu^\ep_x
  +
  bh(u^\ep_x,u^\ep_x)u^\ep_x
& \quad\text{in}\quad
  (0,\infty){\times}\mathbb{T},
\label{equation:pde-ep}
\\
  u^\ep(0,x)
  =
  u_0(x)
& \quad\text{in}\quad
  \mathbb{T},
\label{equation:data-ep}
\end{align}
where $\ep\in(0,1]$ is a parameter. 
The existence of solutions to 
\eqref{equation:pde-ep}-\eqref{equation:data-ep} 
was proved as follows.
\begin{lemma}[{\cite[Proposition~3.1]{onodera1}}]
\label{theorem:onnagurui}
Let $k$ be a positive integer satisfying $k\geqslant2$. 
Then, for any $u_0{\in}H^{k+1}(\mathbb{T};TN)$, there exists 
$T_\ep=T(\ep,\lVert{u}\rVert_{H^{3}})>$ such that 
{\rm \eqref{equation:pde-ep}-\eqref{equation:data-ep}} 
possesses a unique solution 
$u^\ep{\in}C([0,T_\ep];H^{k+1}(\mathbb{T};TN))$.  
\end{lemma}
The proof of Lemma~\ref{theorem:onnagurui} given in \cite{onodera1} 
does not depend on the K\"ahler condition at all. 
Lemma~\ref{theorem:onnagurui} is proved by the standard arguments: 
the contraction mapping theorem and some kind of maximum principle. 
Firstly, we push forward
\eqref{equation:pde-ep}-\eqref{equation:data-ep} 
into $\mathbb{R}^d$ by the Nash embedding $w$, 
and construct a solution taking values 
in a small tubular neighborhood of $w(N)$. 
Secondly, we check that the value of the solution remains in $w(N)$. 
See \cite[Section~3]{onodera1} for the detail. 
\par
We split the proof of Theorem~\ref{theorem:main} into three steps. 
Firstly, we construct a solution by 
the uniform energy estimates and the standard compactness argument. 
Secondly, we check the uniqueness of solutions. 
Finally, we recover the continuity in time of solutions. 
\begin{proof}[Construction of Solutions]
Let $u^\ep$ be a unique solution to 
\eqref{equation:pde-ep}-\eqref{equation:data-ep} 
with a parameter $\ep\in(0,1]$. 
It suffices to show that 
there exists $T>0$ which is independent of $\ep\in(0,1]$,  
such that $\{u^\ep\}_{\ep\in(0,1]}$ is bounded in 
$L^\infty(0,T;H^{k+1}(\mathbb{T};TN))$, 
which is the set of all $H^{k+1}$-valued 
essentially bounded functions on $(0,T)$. 
Indeed, if this is true, then the standard compactness argument 
shows that there exist $u$ 
and a subsequence $\{u^\ep\}_{\ep\in(0,1]}$ such that 
\begin{align*}
  u^\ep \longrightarrow u
& \quad\text{in}\quad
  C([0,T];H^k(\mathbb{T};TN)),
\\
  u^\ep \longrightarrow u
& \quad\text{in}\quad
  L^\infty(0,T;H^{k+1}(\mathbb{T};TN))
  \quad\text{weakly star},
\end{align*}
as $\ep\downarrow0$, 
and $u$ solves \eqref{equation:pde}-\eqref{equation:data} 
and is $H^{k+1}$-valued weakly continuous in time. 
\par
Set $u=u^\ep$ for short. 
Any confusion will not occur. 
We actually evaluate 
$$
\mathcal{N}_{k+1}(u)^2
=
\lVert{u}\rVert_{H^k}^2
+
\lVert{\Lambda\nabla_x^ku_x}\rVert^2, 
$$
where 
$\Lambda=\Lambda_\ep(t,x,u)$ 
is a bounded pseudodifferential operator acting on 
$\Gamma((u^\ep)^{-1}TN)$ defined later, 
and $\lVert\cdot\rVert$ is a norm of $L^2(\mathbb{T};TN)$ defined by 
$$
\lVert{V}\rVert^2
=
\int_\mathbb{T}
h(V,V)dx
\quad\text{for}\quad
V:\mathbb{T}\rightarrow{TN}. 
$$
Set 
$$
T_\ep^\ast
=
\sup\{
T>0
\ \vert \ 
\mathcal{N}_{k+1}(u(t))\leqslant2\mathcal{N}_{k+1}(u_0)
\ \text{for}\ 
t\in[0,T]
\}.
$$
We need to compute 
$$
\nabla_x^{l+1}
\left(
u_t
+
\ep\nabla_x^3u_x
-
a\nabla_x^2u_x
-
J_u\nabla_xu_x
-
bh(u_x,u_x)u_x
\right)
=0, 
\quad
l=0,\dotsc,k.
$$
Main tools of the computation are 
\begin{align}
  \nabla_Xdu(Y)
& =
  \nabla_Ydu(X)
  +
  du([X,Y])
  =
  \nabla_Ydu(X),
\label{equation:commutator1}
\\
  \nabla_X\nabla_YV
& =
  \nabla_Y\nabla_XV
  +
  \nabla_{[X,Y]}V
  +
  R\bigl(du(X),du(Y)\bigr)V
\nonumber
\\
& =
  \nabla_Y\nabla_XV
  +
  R\bigl(du(X),du(Y)\bigr)V.
\label{equation:commutator2}
\end{align}
for $X,Y\in\{\p_t,\p_x\}$ and $V\in\Gamma(u^{-1}TN)$. 
We make use of basic techniques of geometric analysis of nonlinear problems. 
See \cite{nishikawa} for instance. 
\par
In view of \eqref{equation:commutator1} and \eqref{equation:commutator2}, 
we have 
\begin{align}
  \nabla_xu_t
& =
  \nabla_tu_x,
\label{equation:kojima1}
\\
  \nabla_x^2u_t
& =
  \nabla_t\nabla_xu_x
  +
  R(u_x,u_t)u_x,
\nonumber
\\
  \nabla_x^{l+1}u_t
& =
  \nabla_t\nabla_x^lu_x
  +
  \sum_{m=0}^{l-1}
  \nabla_x^{l-1-m}
  \left\{
  R(u_x,u_t)\nabla_x^mu_x
  \right\}
\nonumber
\\
& =
  \nabla_t\nabla_x^lu_x
  +
  \sum_{m=0}^{l-1}
  \nabla_x^{l-1-m}
\nonumber
\\
& \quad
  \times
  \left\{
  R\bigl(
  u_x,
  -
  \ep\nabla_x^3u_x
  +
  a\nabla_x^2u_x
  +
  J_u\nabla_xu_x
  +
  bh(u_x,u_x)u_x
  \bigr)
  \nabla_x^mu_x
  \right\}
\nonumber
\\
& =
  \nabla_t\nabla_x^lu_x
  +
  aR(u_x,\nabla_x^{l+1}u_x)
  -
  \ep
  P_{1,l+1}
  -
  Q_{1,l+1},
\label{equation:kojima2}
\\
  P_{1,l+1}
& =
  \sum_{m=0}^{l-1}
  \nabla_x^{l-1-m}
  \left\{
  R(u_x,\nabla_x^3u_x)\nabla_x^mu_x
  \right\},
\nonumber
\\
  Q_{1,l+1}
& =
  -a
  \sum_{m=0}^{l-1}
  \nabla_x^{l-1-m}
  \left\{
  R(u_x,\nabla_x^2u_x)\nabla_x^mu_x
  \right\}
  +
  aR(u_x,\nabla_x^{l+1}u_x)
\nonumber
\\
& \quad
  -
  \sum_{m=0}^{l-1}
  \nabla_x^{l-1-m}
  \left\{
  R\bigl(
  u_x,
  J_u\nabla_xu_x
  +
  bh(u_x,u_x)u_x
  \bigr)
  \nabla_x^mu_x
  \right\}. 
\nonumber
\end{align}
The Sobolev embeddings show that 
\begin{equation}
\lVert{P_{1,l+1}}\rVert
\leqslant
C_k
\lVert{u}\rVert_{H^{l+3}},
\quad
\lVert{Q_{1,l+1}}\rVert
\leqslant
C_k
\lVert{u}\rVert_{H^{l+1}}
\label{equation:kojima3}
\end{equation}
for $t\in[0,T_\ep^\ast]$, 
where $C_k>1$ is a constant depending only on 
$a$, $b$ and $\lVert{u_0}\rVert_{H^{k+1}}$ and not on $\ep\in(0,1]$. 
Such constants are denoted by the same notation $C_k$ below. 
Using \eqref{equation:commutator1} and \eqref{equation:commutator2} again, 
we have 
\begin{align}
  \nabla_x^{l+1}(J_u\nabla_xu_x)
& =
  \nabla_xJ_u\nabla_x\nabla_x^lu_x
  +
  l(\nabla_xJ_u)\nabla_x\nabla_x^lu_x
  +
  Q_{2,l+1},
\label{equation:kojima4}
\\
  \nabla_x^{l+1}\bigl\{h(u_x,u_x)u_x\bigr\}
& =
  h(u_x,u_x)\nabla_x\nabla_x^lu_x
  +
  2\bigl\{h(\nabla_x^lu_x,u_x)\bigr\}_xu_x
  +
  Q_{3,l+1},
\label{equation:kojima5}
\end{align}
\begin{align*}
  Q_{2,l+1}
& =
  \sum_{m=0}^{l-1}
  \frac{(l+1)!}{m!(l+1-m)!}
  (\nabla_x^{l+1-m}J_u)
  \nabla_x^{m+1}u_x,
\\
  Q_{3,l+1}
& =
  \sum_{\substack{\alpha+\beta+\gamma=l+1 \\ \alpha,\beta,\gamma\leqslant{l}}}
  \frac{(l+1)!}{\alpha!\beta!\gamma!}
  h(\nabla_x^\alpha{u_x},\nabla_x^\beta{u_x})\nabla_x^\gamma{u_x}, 
  -
  2h(\nabla_x^lu_x,\nabla_xu_x)u_x.
\end{align*}
$Q_{2,l+1}$ and $Q_{3,l+1}$ have the same estimates as $Q_{1,l+1}$. 
Combining 
\eqref{equation:kojima1}, 
\eqref{equation:kojima2}, 
\eqref{equation:kojima3}, 
\eqref{equation:kojima4} 
and 
\eqref{equation:kojima5}, 
we obtain 
\begin{align}
& \bigl\{
  \nabla_t
  +
  \ep\nabla_x^4
  -
  a\nabla_x^3
  -
  \nabla_xJ_u\nabla_x
  -
  l(\nabla_xJ_u)\nabla_x
  -
  bh(u_x,u_x)\nabla_x  
  \bigr\}\nabla_x^lu_x
\nonumber
\\
  =
& -
  aR(u_x,\nabla_x^{l+1}u_x)u_x
  +
  2b\bigl\{h(\nabla_x^lu_x,u_x)\bigr\}_xu_x
  +
  {\ep}P_{l+1}
  +
  Q_{l+1},
\label{equation:kojima6}
\end{align}
\begin{equation}
\lVert{P_{l+1}}\rVert
\leqslant
C_k
\lVert{u}\rVert_{H^{l+3}},
\quad
\lVert{Q_{l+1}}\rVert
\leqslant
C_k
\lVert{u}\rVert_{H^{l+1}}
\quad\text{for}\quad
t\in[0,T_\ep^\ast].
\label{equation:kojima7} 
\end{equation}
\par
By using \eqref{equation:kojima6}, we have 
\begin{align}
  \frac{d}{dt}
  \lVert{u}\rVert_{H^k}^2
  =
& 2
  \sum_{l=0}^{k-1}
  \int_\mathbb{T}
  h(\nabla_t\nabla_x^lu_x,\nabla_x^lu_x)
  dx
\nonumber
\\
  =
& -2\ep
  \sum_{l=0}^{k-1}
  \int_\mathbb{T}
  h(\nabla_x^4\nabla_x^lu_x,\nabla_x^lu_x)
  dx
\label{equation:yoshio1}
\\
& +2a
  \sum_{l=0}^{k-1}
  \int_\mathbb{T}
  h(\nabla_x^3\nabla_x^lu_x,\nabla_x^lu_x)
  dx
\label{equation:yoshio2}
\\
& +2
  \sum_{l=0}^{k-1}
  \int_\mathbb{T}
  h(\nabla_xJ_u\nabla_x\nabla_x^lu_x,\nabla_x^lu_x)
  dx
\label{equation:yoshio3}
\\
& +2
  \sum_{l=0}^{k-1}
  \int_\mathbb{T}
  h\bigl((\nabla_xJ_u)\nabla_x\nabla_x^lu_x,\nabla_x^lu_x\bigr)
  dx
\label{equation:yoshio4}
\\
& +2b
  \sum_{l=0}^{k-1}
  \int_\mathbb{T}
  h(u_x,u_x)
  h(\nabla_x\nabla_x^lu_x,\nabla_x^lu_x)
  dx
\label{equation:yoshio5}
\\
& -2a
  \sum_{l=0}^{k-1}
  \int_\mathbb{T}
  h\bigl(R(u_x,\nabla_x^{l+1}u_x)u_x,\nabla_x^lu_x\bigr)
  dx
\label{equation:yoshio6}
\\
& +4b
  \sum_{l=0}^{k-1}
  \int_\mathbb{T}
  \bigl\{h(\nabla_x^lu_x,u_x)\bigr\}_x
  h(u_x,\nabla_x^lu_x)
  dx
\label{equation:yoshio7}
\\
& +2
  \sum_{l=0}^{k-1}
  \int_\mathbb{T}
  h(\ep{P_{l+1}}+Q_{l+1},\nabla_x^lu_x)
  dx.
\label{equation:yoshio8}
\end{align}
Using integration by parts and the properties of $h$ and $J$, 
we deduce that 
\eqref{equation:yoshio1}, 
\eqref{equation:yoshio2}, 
\eqref{equation:yoshio3}, 
\eqref{equation:yoshio5}, 
\eqref{equation:yoshio7} 
respectively become 
\begin{align}
  \text{\eqref{equation:yoshio1}}
& =
  -2\ep
  \sum_{l=0}^{k-1}
  \int_\mathbb{T}
  h(\nabla_x^{l+2}u_x,\nabla_x^{l+2}u_x)
  dx,  
\label{equation:yoshio11}
\\
  \text{\eqref{equation:yoshio2}}
& =
  -2a
  \sum_{l=0}^{k-1}
  \int_\mathbb{T}
  h(\nabla_x\nabla_x^{l+1}u_x,\nabla_x^{l+1}u_x)
  dx
\nonumber
\\
& =
  -a
  \sum_{l=0}^{k-1}
  \int_\mathbb{T}
  \bigl\{h(\nabla_x^{l+1}u_x,\nabla_x^{l+1}u_x)\bigr\}_x
  dx
  =
  0,
\label{equation:yoshio12}
\\
  \text{\eqref{equation:yoshio3}}
& =
  -2
  \sum_{l=0}^{k-1}
  \int_\mathbb{T}
  h(J_u\nabla_x^{l+1}u_x,\nabla_x^{l+1}u_x)
  dx
  =0,
\label{equation:yoshio13}
\\
  \text{\eqref{equation:yoshio5}}
& =
  b
  \sum_{l=0}^{k-1}
  \int_\mathbb{T}
  h(u_x,u_x)
  \bigl\{h(\nabla_x^lu_x,\nabla_x^lu_x)\bigr\}_x
  dx
\nonumber
\\
& =
 -b
  \sum_{l=0}^{k-1}
  \int_\mathbb{T}
  \{h(u_x,u_x)\}_x
  h(\nabla_x^lu_x,\nabla_x^lu_x)
  dx,
\label{equation:yoshio05}
\\
  \text{\eqref{equation:yoshio7}}
& =
  2b
  \sum_{l=0}^{k-1}
  \int_\mathbb{T}
  \bigl\{h(\nabla_x^lu_x,u_x)^2\bigr\}_x
  dx
  =
  0. 
\label{equation:yoshio17}
\end{align}
Recall the property of the Riemannian curvature tensor $R$: 
$h(R(X,Y)Z,W)=h(R(Z,W)X,Y)$ for any vector fields $X,Y,X,W$ on $N$. 
Using this and integration by parts, we deduce 
\begin{align*}
  \text{\eqref{equation:yoshio6}}
& =
  -2a
  \sum_{l=0}^{k-1}
  \int_\mathbb{T}
  h\bigl(R(u_x,\nabla_x^lu_x)u_x,\nabla_x^{l+1}u_x\bigr)
  dx
\\
& =
  2a
  \sum_{l=0}^{k-1}
  \int_\mathbb{T}
  h\bigl(R(u_x,\nabla_x^{l+1}u_x)u_x,\nabla_x^lu_x\bigr)
  dx
\\
& +2a
  \sum_{l=0}^{k-1}
  \int_\mathbb{T}
  h\bigl((\nabla^NR)(u_x,u_x,\nabla_x^lu_x)u_x,\nabla_x^lu_x\bigr)
  dx
\\
& +2a
  \sum_{l=0}^{k-1}
  \int_\mathbb{T}
  h\bigl(R(\nabla_xu_x,\nabla_x^lu_x)u_x,\nabla_x^lu_x\bigr)
  dx
\\
& +2a
  \sum_{l=0}^{k-1}
  \int_\mathbb{T}
  h\bigl(R(u_x,\nabla_x^lu_x)\nabla_xu_x,\nabla_x^lu_x\bigr)
  dx,
\end{align*}
which implies 
\begin{align}
  \text{\eqref{equation:yoshio6}}
& =
  a
  \sum_{l=0}^{k-1}
  \int_\mathbb{T}
  h\bigl((\nabla^NR)(u_x,u_x,\nabla_x^lu_x)u_x,\nabla_x^lu_x\bigr)
  dx
\nonumber
\\
& +a
  \sum_{l=0}^{k-1}
  \int_\mathbb{T}
  h\bigl(R(\nabla_xu_x,\nabla_x^lu_x)u_x,\nabla_x^lu_x\bigr)
  dx
\nonumber
\\
& +a
  \sum_{l=0}^{k-1}
  \int_\mathbb{T}
  h\bigl(R(u_x,\nabla_x^lu_x)\nabla_xu_x,\nabla_x^lu_x\bigr)
  dx. 
\label{equation:yoshio06}
\end{align}
Applying the Schwarz inequality to 
\eqref{equation:yoshio05}, 
\eqref{equation:yoshio06} 
and 
\eqref{equation:yoshio8}, 
we have 
\begin{equation}
\lvert\text{\eqref{equation:yoshio5}}\rvert, 
\lvert\text{\eqref{equation:yoshio6}}\rvert
\leqslant
C_k\lVert{u}\rVert_{H^{k}}^2,
\label{equation:yoshio15}
\end{equation}
\begin{align}
  \lvert\text{\eqref{equation:yoshio8}}\rvert
& \leqslant
  C_k\ep\lVert{u}\rVert_{H^{k+2}}\lVert{u}\rVert_{H^k}
  +
  C_k\lVert{u}\rVert_{H^{k}}^2 
\nonumber
\\
& \leqslant
  2\ep
  \sum_{l=0}^{k-1}
  \int_\mathbb{T}
  h(\nabla_x^{l+2}u_x,\nabla_x^{l+2}u_x)
  dx
  +
  C_k\lVert{u}\rVert_{H^{k}}^2. 
\label{equation:yoshio18}
\end{align}
Similarly, \eqref{equation:yoshio4} is estimated as 
\begin{equation}
\lvert\text{\eqref{equation:yoshio4}}\rvert
\leqslant
C_k\lVert{u}\rVert_{H^{k+1}}\lVert{u}\rVert_{H^{k}},
\label{equation:yoshio14}
\end{equation}
Combining 
\eqref{equation:yoshio11}, 
\eqref{equation:yoshio12}, 
\eqref{equation:yoshio13}, 
\eqref{equation:yoshio17}, 
\eqref{equation:yoshio15}, 
\eqref{equation:yoshio18} 
and  
\eqref{equation:yoshio14}, 
we obtain
\begin{equation}
\frac{d}{dt}
\lVert{u}\rVert_{H^k}^2
\leqslant
C_k\lVert{u}\rVert_{H^{k+1}}\lVert{u}\rVert_{H^{k}}. 
\label{equation:shiota}
\end{equation}
\par
Next we estimate $\Lambda\nabla_x^ku_x$. 
Here we define the pseudodifferential operator $\Lambda$. 
Let $\{N_\alpha\}$ be 
the set of local coordinate neighborhood of $N$, 
and let $y_\alpha^1,\dotsc,y_\alpha^{2n}$ be the local coordinates of $N_\alpha$. 
Pick up a partition of unity $\{\Phi_\alpha\}$ subordinated to $\{N_\alpha\}$, 
and pick up $\{\Psi_\alpha\}{\subset}C^\infty_0(N)$ so that 
$$
\Psi_\alpha=1 
\quad\text{in}\quad
\operatorname{supp}[\Phi_\alpha], 
\quad
\operatorname{supp}[\Psi_\alpha]\subset{N_\alpha},
$$
where $C^\infty_0(N)$ is the set of 
all compactly supported $C^\infty$-functions on $N$. 
We define a properly supported pseudodifferential operator $\Lambda$ 
acting on $\Gamma(u^{-1}TN)$ by 
$$
\Lambda=1-\tilde{\Lambda}, 
\quad
\tilde{\Lambda}
=
\frac{\sqrt{-1}k}{3a}J_u
\sum_{\alpha}
\Phi_\alpha(u)p(D_x)\Psi_\alpha(u). 
$$
If 
$$
V(x)
=
\sum_{a=1}^{2n}
V^a(x)
\left(\frac{\p}{\p{y_\alpha^a}}\right)_u
\in\Gamma(u^{-1}TN)
$$ 
is supported in $u^{-1}(N_\alpha)$, then 
$$
\Phi_\alpha(u)p(D_x)V(x)
=
\sum_{a=1}^{2n}
\left\{\Phi_\alpha(u)p(D_x)V^a(x)\right\}
\left(\frac{\p}{\p{y_\alpha^a}}\right)_u
$$
is well-defined and supported in $u^{-1}(N_\alpha)$. 
Then, each term in $\tilde{\Lambda}$ can be treated 
as a pseudodifferential operator acting on $\mathbb{R}^d$-valued functions, 
and we can make use of pseudodifferential operators with nonsmooth symbols. 
In other words, we can deal with $\tilde{\Lambda}$ 
as if it were a pseudodifferential operator with a smooth symbol. 
See \cite[Section~2]{chihara1} and \cite{nagase} for the detail. 
Symbolic calculus below is valid 
since the Sobolev embedding shows that 
$u(t){\in}C^{4+\delta}(\mathbb{T})$ for $\delta\in(0,1/2)$. 
It is easy to see that there exists $C_k>1$ such that 
$$
C_k^{-1}\mathcal{N}_{k+1}(u)
\leqslant
\lVert{u}\rVert_{H^{k+1}}
\leqslant
C_k\mathcal{N}_{k+1}(u)
\quad\text{for}\quad
t\in[0,T_\ep^\ast].
$$
\par
We compute 
\begin{align*}
  0
& =
  \Lambda
  \nabla_x^{k+1}
  \left(
  u_t
  +
  \ep\nabla_x^3u_x
  -
  a\nabla_x^2u_x
  -
  J_u\nabla_xu_x 
  -
  bh(u_x,u_x)u_x
  \right)
\\
& =
  \Lambda
  \bigl\{
  \nabla_t
  +
  \ep\nabla_x^4
  -
  a\nabla_x^3
  -
  \nabla_xJ_u\nabla_x
  -
  k(\nabla_xJ_u)\nabla_x
  -
  bh(u_x,u_x)\nabla_x  
  \bigr\}\nabla_x^ku_x
\\
& -
  \Lambda
  \left\{
  -
  aR(u_x,\nabla_x^{k+1}u_x)u_x
  +
  2b\bigl\{h(\nabla_x^ku_x,u_x)\bigr\}_xu_x
  +
  {\ep}P_{k+1}
  +
  Q_{k+1}
  \right\}.
\end{align*}
A direct computation shows that 
\begin{equation}
\Lambda\nabla_t
=
\nabla_t\Lambda
-
\frac{\p\Lambda}{\p{t}}
=
\nabla_t\Lambda
+
\frac{\p\tilde{\Lambda}}{\p{t}}, 
\label{equation:opp1}
\end{equation}
$$
\left\lVert
\frac{\p\tilde{\Lambda}}{\p{t}}
\nabla_x^ku_x
\right\rVert
\leqslant
C_k\lVert{u}\rVert_{H^k}.
$$
Let $I_{2n}$ be the $2n\times2n$ identity matrix. 
If we use a local expression 
$\nabla_x^4=\p_x^4+A_3\p_x^3+A_2\p_x^2+A_1\p_x+A_0$ 
with $2n\times2n$ matrices $A_j$, $j=0,1,2,3$, 
we deduce that 
\begin{equation}
\ep\Lambda\nabla_x^4
=
\ep\nabla_x^4\Lambda
+
\ep
[\Lambda,\nabla_x^4]
=
\ep\nabla_x^4\Lambda
-
\ep
[\tilde{\Lambda},\nabla_x^4], 
\label{equation:opp2}
\end{equation}
$$
[\tilde{\Lambda},\nabla_x^4]
=
\left[
\frac{\sqrt{-1}k}{3a}J_up(D_x),I_{2n}\p_x^4+\dotsb
\right],
\quad
\lVert[\tilde{\Lambda},\nabla_x^4]\nabla_x^ku_x\rVert
\leqslant
C_k\lVert{u}\rVert_{H^{k+3}},
$$
since the matrices of principal symbols 
$J_up(D_x)$ and $\nabla_x^4$ commute with each other. 
Next computation is the most crucial part 
of the proof of Theorem~\ref{theorem:main}. 
In the same way as $\ep\Lambda\nabla_x^4$, we have 
$$
-a\Lambda\nabla_x^3
=
-a\nabla_x^3\Lambda
+
a[\tilde{\Lambda},\nabla_x^3]. 
$$
We see the commutator above in detail. 
A direct computation shows that  
\begin{align*}
  a[\tilde{\Lambda},\nabla_x^3]
& =
  \frac{\sqrt{-1}k}{3}
  \sum_{\alpha}
  J_u\Phi_\alpha(u)p(D_x)\Psi_\alpha(u)\nabla_x^3
\\
& -
  \frac{\sqrt{-1}k}{3}
  \sum_{\alpha}
  \nabla_x^3J_u\Phi_\alpha(u)p(D_x)\Psi_\alpha(u)
\\
& =
  \frac{\sqrt{-1}k}{3}
  \sum_{\alpha}
  J_u\Phi_\alpha(u)p(D_x)\nabla_x^3\Psi_\alpha(u)
\\
& -
  \frac{\sqrt{-1}k}{3}
  \sum_{\alpha}
  \nabla_x^3J_u\Phi_\alpha(u)p(D_x)\Psi_\alpha(u)
\\
& +
  \frac{\sqrt{-1}k}{3}
  \sum_{\alpha}
  J_u\Phi_\alpha(u)p(D_x)[\Psi_\alpha(u),\nabla_x^3]. 
\end{align*}
The last term above is a smoothing operator since 
$
\operatorname{supp}[\{\Psi_\alpha(u)\}_x]
\cap
\operatorname{supp}[\Phi_\alpha(u)]
=
\emptyset
$. 
If we compute the commutator in the framework of 
modulo $L^2$-bounded operators, we deduce   
\begin{align*}
  a[\tilde{\Lambda},\nabla_x^3]
& \equiv
  \frac{\sqrt{-1}k}{3}
  \sum_{\alpha}
  \left\{
  J_u\Phi_\alpha(u)p(D_x)\nabla_x^3
  -
  \nabla_x^3J_u\Phi_\alpha(u)p(D_x)
  \right\}
  \Psi_\alpha(u)
\\
& =
  \frac{\sqrt{-1}k}{3}
  \sum_{\alpha}
  J_u\Phi_\alpha(u)[p(D_x),\nabla_x^3]\Psi_\alpha(u)
\\
& -\sqrt{-1}k
  \sum_{\alpha}
  \bigl[\nabla_x\{J_u\Phi_\alpha(u)\}\bigr]
  \nabla_x^2p(D_x)\Psi_\alpha(u)
\\
& -\sqrt{-1}k
  \sum_{\alpha}
  \bigl[\nabla_x^2\{J_u\Phi_\alpha(u)\}\bigr]
  \nabla_xp(D_x)\Psi_\alpha(u)
\\
& \equiv
  -\sqrt{-1}k
  \sum_{\alpha}
  \bigl[\nabla_x\{J_u\Phi_\alpha(u)\}\bigr]
  \nabla_x^2p(D_x)\Psi_\alpha(u)
\\
& =
  k
  \sum_{\alpha}
  \bigl[\nabla_x\{J_u\Phi_\alpha(u)\}\bigr]
  \nabla_x^2\frac{p(D_x)}{\sqrt{-1}}\Psi_\alpha(u)
\\
& \equiv
  k
  \sum_{\alpha}
  \bigl[\nabla_x\{J_u\Phi_\alpha(u)\}\bigr]
  \nabla_x\Psi_\alpha(u)
\\
& =
  k
  \sum_{\alpha}
  \bigl[\nabla_x\{J_u\Phi_\alpha(u)\}\bigr]\nabla_x
\\
& =
  k(\nabla_xJ_u)\sum_{\alpha}\Phi_\alpha(u)\nabla_x
  +
  kJ_u\left\{\sum_{\alpha}\Phi_\alpha(u)\right\}_x\nabla_x
\\
& =
  k(\nabla_xJ_u)\nabla_x.
\end{align*}
Thus, 
\begin{equation}
-a\Lambda\nabla_x^3
\equiv
-a\nabla_x^3\Lambda
+
k(\nabla_xJ_u)\nabla_x
\label{equation:opp3} 
\end{equation}
modulo $L^2$-bounded operators. 
In the same way as above, we deduce 
\begin{align}
  \Lambda\nabla_xJ_u\nabla_x
& \equiv
  \nabla_xJ_u\nabla_x\Lambda, 
\label{equation:opp4}
\\
  -k\Lambda(\nabla_xJ_u)\nabla_x
& \equiv
  -k(\nabla_xJ_u)\nabla_x,
\label{equation:opp5}
\\
  -b\Lambda{h(u_x,u_x)\nabla_x}
& \equiv
  -bh(u_x,u_x)\nabla_x\Lambda
\label{equation:opp6} 
\end{align}
modulo $L^2$-bounded operators. 
By using $1=\Lambda+\tilde{\Lambda}$, we deduce 
\begin{align}
& -
  a\Lambda\left\{R(u_x,\nabla_x^{k+1}u_x)u_x\right\}
\nonumber
\\
  =
& -
  aR(u_x,\nabla_x^{k+1}u_x)u_x
  +
  a\tilde{\Lambda}\left\{R(u_x,\nabla_x^{k+1}u_x)u_x\right\}
\nonumber
\\
  =
& -
  aR(u_x,\nabla_x\Lambda\nabla_x^ku_x)u_x
  -
  aR(u_x,\nabla_x\tilde{\Lambda}\nabla_x^ku_x)u_x
\nonumber
\\
& +
  a\tilde{\Lambda}\nabla_x\left\{R(u_x,\nabla_x^ku_x)u_x\right\}
  -
  a\tilde{\Lambda}\left\{(\nabla^NR)(u_x,u_x,\nabla_x^ku_x)\right\}
\nonumber
\\
& -
  a\tilde{\Lambda}\left\{R(\nabla_xu_x,\nabla_x^ku_x)u_x\right\}
  -
  a\tilde{\Lambda}\left\{R(u_x,\nabla_x^ku_x)\nabla_xu_x\right\}
\nonumber
\\
  =
& -
  aR(u_x,\nabla_x\Lambda\nabla_x^ku_x)u_x
  +
  Q^\prime_{1,k+1},
\label{equation:opp7}
\end{align}
\begin{align}
& 2b
  \Lambda
  \bigl[
  \bigl\{
  h(\nabla_x^ku_x,u_x)
  \bigr\}_x
  u_x
  \bigr]
\nonumber
\\
  =
& 2b
  \bigl\{
  h(\nabla_x^ku_x,u_x)
  \bigr\}_x
  u_x
  -
  2b
  \tilde{\Lambda}
  \bigl[
  \bigl\{
  h(\nabla_x^ku_x,u_x)
  \bigr\}_x
  u_x
  \bigr]
\nonumber
\\
  =
& 2b
  \bigl\{
  h(\Lambda\nabla_x^ku_x,u_x)
  \bigr\}_x
  u_x
  +
  2b
  \bigl\{
  h(\tilde{\Lambda}\nabla_x^ku_x,u_x)
  \bigr\}_x
  u_x
\nonumber
\\
& -
  2b
  \tilde{\Lambda}
  \nabla_x
  \bigl\{
  h(\nabla_x^ku_x,u_x)u_x
  \bigr\}
  +
  2b
  \tilde{\Lambda}
  \bigl\{
  h(\nabla_x^ku_x,u_x)
  \nabla_xu_x
  \bigr\}
\nonumber
\\
  =
& 2b
  \bigl\{
  h(\Lambda\nabla_x^ku_x,u_x)
  \bigr\}_x
  u_x
  +
  Q^\prime_{2,k+1},
\label{equation:opp8}
\end{align}
$$
\lVert{Q^\prime_{1,k+1}}\rVert, 
\lVert{Q^\prime_{2,k+1}}\rVert 
\leqslant 
C_k
\lVert{u}\rVert_{H^{k+1}}.
$$
Combining 
\eqref{equation:opp1}, 
\eqref{equation:opp2}, 
\eqref{equation:opp3}, 
\eqref{equation:opp4}, 
\eqref{equation:opp5}, 
\eqref{equation:opp6}, 
\eqref{equation:opp7} 
and 
\eqref{equation:opp8}, 
we obtain
\begin{align}
& \bigl\{
  \nabla_t
  +
  \ep\nabla_x^4
  -
  a\nabla_x^3
  -
  \nabla_xJ_u\nabla_x
  -
  bh(u_x,u_x)\nabla_x  
  \bigr\}
  \Lambda\nabla_x^ku_x
\nonumber
\\
  =
& -
  aR(u_x,\nabla_x\Lambda\nabla_x^ku_x)u_x
  +
  2b\bigl\{h(\Lambda\nabla_x^ku_x,u_x)\bigr\}_xu_x
  +
  {\ep}P^\prime_{k+1}
  +
  Q^\prime_{k+1},
\label{equation:lastapia1}
\end{align}
\begin{equation}
\lVert{P^\prime_{l+1}}\rVert
\leqslant
C_k
\lVert{u}\rVert_{H^{k+3}},
\quad
\lVert{Q^\prime_{k+1}}\rVert
\leqslant
C_k
\lVert{u}\rVert_{H^{k+1}}
\quad\text{for}\quad
t\in[0,T_\ep^\ast].
\label{equation:lastapia2} 
\end{equation}
Here we remark that 
$-k(\nabla_xJ_u)\nabla_x$ 
is canceled out in the left hand side of \eqref{equation:lastapia1} 
by $a[\tilde{\Lambda},\nabla_x^3]$. 
By computations similar to 
\eqref{equation:yoshio11}, 
\eqref{equation:yoshio12}, 
\eqref{equation:yoshio13}, 
\eqref{equation:yoshio17}, 
\eqref{equation:yoshio15}, 
\eqref{equation:yoshio18} 
and not to \eqref{equation:yoshio14}, 
we can deduce from 
\eqref{equation:lastapia1} and \eqref{equation:lastapia2} that 
\begin{equation}
\frac{d}{dt}\lVert\Lambda\nabla_x^ku_x\rVert^2
\leqslant
C_k\mathcal{N}_{k+1}(u)^2.
\label{equation:reiko} 
\end{equation}
Combining \eqref{equation:shiota} and \eqref{equation:reiko}, we obtain
\begin{equation}
\frac{d}{dt}\mathcal{N}_{k+1}(u)
\leqslant
C_k
\mathcal{N}_{k+1}(u) 
\quad\text{for}\quad
t\in[0,T_\ep^\ast].
\label{equation:joe}
\end{equation}
If we take $t=T_\ep^\ast$, then we have 
$2\mathcal{N}_{k+1}(u_0)\leqslant\mathcal{N}_{k+1}(u_0)e^{C_kT_\ep^\ast}$, 
which implies 
$T_\ep^\ast{\geqslant}T=\log2/C_k>0$. 
\par
Thus $\{u^\ep\}_{\ep{\in(0,1]}}$ is bounded in 
$L^\infty(0,T;H^{k+1}(\mathbb{T};TN))$. 
This completes the proof. 
\end{proof}
\begin{proof}[Uniqueness of Solutions] 
The uniqueness of solutions was proved in \cite[Section~5]{onodera1}. 
The proof given there does not depend on 
the K\"ahler condition at all. 
We prove the uniqueness by $H^1$-energy estimates of 
the difference of two solutions 
with the same initial data in $\mathbb{R}^d$. 
The symmetry of the second fundamental form of the mapping 
$w{\circ}u$ plays a crucial role. 
See \cite[Section~5]{onodera1} for the detail. 
\end{proof}
\begin{proof}[Recovery of Continuity in Time] 
Let $u{\in}L^\infty(0,T;H^{k+1}(\mathbb{T};TN))$ 
be the unique solution to \eqref{equation:pde}-\eqref{equation:data}. 
Following \cite[Section~3]{chihara3}, 
we prove that $\nabla_x^ku_x$ is strongly continuous in time. 
We remark that we have already known that 
$u{\in}C([0,T];H^k(\mathbb{T};TN))$ 
and  $\nabla_x^ku_x$ is a weakly continuous 
$L^2(\mathbb{T};TN)$-valued function on $[0,T]$. 
We identify $N$ and $w(N)$ below. 
Let $\{u^\ep\}_{\ep\in(0,1]}$ be a sequence of solutions 
to \eqref{equation:pde-ep}-\eqref{equation:data-ep}, 
which approximates $u$. 
We can easily check that for any $\phi{\in}C^\infty([0,T]\times\mathbb{T};\mathbb{R}^d)$, 
\begin{align*}
  \Lambda_\ep^\ast\phi \longrightarrow \Lambda^\ast\phi
& \quad\text{in}\quad
  L^2((0,T){\times}\mathbb{T};\mathbb{R}^d),
\\
  \Lambda_\ep\nabla_x^ku_x^\ep \longrightarrow \tilde{u}
& \quad\text{in}\quad
  L^2((0,T){\times}\mathbb{T};\mathbb{R}^d)
  \quad\text{weakly star},
\end{align*}
as $\ep\downarrow0$ with some $\tilde{u}$. 
Then, $\tilde{u}=\Lambda\nabla_x^ku_x$ in the sense of distributions. 
We denote by $\mathscr{L}(\mathscr{H})$ 
the set of all bounded linear operators 
of a Hilbert space $\mathscr{H}$ to itself. 
The time-continuity of $\nabla_x^ku_x$ is equivalent to 
that of $\Lambda\nabla_x^ku_x$ since 
$\Lambda{\in}C([0,T];\mathscr{L}(L^2(\mathbb{T};\mathbb{R}^d)))$. 
\par
It suffices to show that  
\begin{equation}
\lim_{t\downarrow0}
\Lambda(t)\nabla_x^ku_x(t)
=
\Lambda(0)\nabla_x^ku_{0x}
\quad\text{in}\quad
L^2(\mathbb{T};\mathbb{R}^d),
\label{equation:gillian}  
\end{equation}
since the other cases can be proved in the same way. 
\eqref{equation:joe} and the lower semicontinuity of $L^2$-norm imply 
$$
\sum_{l=0}^{k-1}
\lVert{\nabla_x^lu_x(t)}\rVert^2
+
\lVert{\Lambda(t)\nabla_x^ku_x(t)}\rVert^2
\leqslant
\sum_{l=0}^{k-1}
\lVert{\nabla_x^lu_{0x}}\rVert^2
+
\lVert{\Lambda(0)\nabla_x^ku_{0x}}\rVert^2
+
C_k\mathcal{N}_{k+1}(u_0)^2t
$$
provided that $\ep\downarrow0$. 
Letting $t\downarrow0$, we have 
$$
\limsup_{t\downarrow0}
\lVert{\Lambda(t)\nabla_x^ku_x(t)}\rVert^2
\leqslant
\lVert{\Lambda(0)\nabla_x^ku_{0x}}\rVert^2
$$
which implies \eqref{equation:gillian}. 
This completes the proof.
\end{proof}
%
%

\end{document}